\documentclass[oneside]{amsart}

\usepackage{amsmath}
\usepackage{amssymb}
\usepackage{amsthm}
\usepackage{color}
\usepackage{graphicx}
\usepackage{tikz}
\usepackage{pinlabel}
\usetikzlibrary{calc}
\usepackage{xcolor}
\usepackage[hidelinks]{hyperref}
\usepackage[alphabetic,msc-links]{amsrefs}
\usepackage[capitalize]{cleveref}

\definecolor{figred}{cmyk}{0,.8,1,0}
\definecolor{figyellow}{cmyk}{.1,.05,.9,0}
\definecolor{figorange}{cmyk}{0,.5,1,0}
\definecolor{figblue}{cmyk}{1,.5,0,0}
\definecolor{figmagenta}{cmyk}{.1,.7,0,0}
\definecolor{figcyan}{cmyk}{.8,0,0,0}
\definecolor{figgreen}{cmyk}{.97,0,.75,0}

\newcommand{\sph}{\mathcal{S}}

\makeatletter
\newtheorem*{rep@theorem}{\rep@title}
\newcommand{\newreptheorem}[2]{%
\newtheorem*{rep@#1}{\rep@title}%
\newenvironment{rep#1}[1]{%
 \def\rep@title{#2 \ref*{##1}}%
\begin{rep@#1}}%
{\end{rep@#1}}}
\makeatother

\newtheorem{theor}{Theorem} 
\newreptheorem{theor}{Theorem}

\newtheorem{corol}[theor]{Corollary} 
\newreptheorem{corol}{Corollary}

\newtheorem{theorem}{Theorem}[section]
\newtheorem{lemma}[theorem]{Lemma}
\theoremstyle{definition}

\title{Coloring spheres in 3--manifolds}
\author{Edgar A. Bering IV}
\address{San Jos\'{e} State University, One Washington Square, San Jose, CA, 95112}
\email{edgar.bering@sjsu.edu}
\author{Bennett Haffner}
\email{bennett.haffner@sjsu.edu}
\author{Estephanie Ortiz}
\email{estephanie.ortiz@sjsu.edu}
\author{Olivia Sanchez}
\email{olivia.sanchez@sjsu.edu}

\begin{document}

\begin{abstract}
The sphere graph of $M_r$, a connect sum of $r$ copies of $S^1\times S^2$ was introduced by Hatcher as an analog of the curve graph of a surface to study the outer automorphism group of a free group $F_r$. Bestvina, Bromberg, and Fujiwara proved that the chromatic number of the curve graph is finite; bounds were subsequently improved by Gaster, Greene, and Vlamis. Motivated by the analogy, we provide upper and lower bounds for the chromatic number of the sphere graph of $M_r$. As a corollary to the prime decomposition of 3--manifolds, this gives bounds on the chromatic number of the sphere graph for any orientable 3--manifold.
\end{abstract}

\maketitle

\section{Introduction}

The sphere graph of a connect sum of $r$ copies of $S^1\times S^2$ was introduced by Hatcher to investigate homological stability of the automorphism groups of free groups; analogous to Harvey's use of the curve graph of a surface to investigate mapping class groups of surfaces~\cite{hatcher-sphere}. Subsequently, the large scale geometry of both graphs have been extensively studied on their respective sides of the long-running analogy between automorphisms of free groups and mapping class groups of surfaces~\cites{mm1,mm2,BBF,hm1,hm2,bf-ff,subfactor}.
In addition to their above mentioned role, both graphs are rich combinatorial objects. The graph theory of both the curve graph~\cites{ivanov,al1,al2,GGV,mrt,dkg} and the sphere graph~\cites{free splittings,bering-leininger} have seen a similar development of parallel results.

We will focus our attention on the chromatic number. Bestvina, Bromberg, and Fujiwara first proved that the chromatic number of the curve graph is finite, on the way to determining the asymptotic dimension of the mapping class group of a surface~\cite{BBF}*{Lemma 5.6}. Their initial bound was doubly exponential in the genus of the surface. Subsequently, Gaster, Greene, and Vlamis produced a lower bound super-linear in the genus and a singly exponential exponential upper bound~\cite{GGV}*{Theorem 1.5}. Motivated by the guiding analogy, we obtain upper and lower bounds for the sphere graph of a connect sum of $r$ copies of $S^1\times S^2$.

\begin{theor}\label{maintheorem}
\[ r\log r \preceq \chi(\sph(M_r)) \preceq 2^{9r2^r}. \]
\end{theor}

The proof of \cref{maintheorem} is given in two lemmas: \cref{lower-bound} for the lower bound and \cref{upper-bound} for the upper bound.

Prime decomposition of 3--manifolds~\cite{hempel}*{Theorem 3.15} allows us to apply \cref{maintheorem} to an orientable 3--manifold $M$. Let $M = N_1\sharp N_2\sharp\cdots\sharp N_k\sharp M_r$ be the prime decomposition of $M$ into a connect sum of irreducible 3--manifolds $N_i$ and $r$ copies of $S^1\times S^2$ collected in the $M_r$ term. An essential embedded sphere in $M$ is isotopic to either a connect-sum sphere or an essential sphere in $M_r$. As there are finitely many connect-sum spheres we obtain the following corollary.

\begin{corol}
If $M$ is an orientable 3--manifold with $r$ copies of $S^1\times S^2$ in its prime decomposition then $r\log(r)\preceq\chi(\sph(M)) \preceq 2^{9r2^r}$.
\end{corol}

The proof of the upper bound in \cref{maintheorem} involves a coloring constructed using double covers and $\mathbb{Z}_2$ homology. Bestvina and Feighn define a coloring on the \emph{free factor} graph of a free group~\cite{subfactor}*{Definition 4.11, Example 4.12}, which can be characterized in terms of $\mathbb{Z}_2$ homology. While not explicitly calculated, the coloring provided by Bestvina and Feighn is also doubly-exponential in the rank. Coloring the free factor complex is insufficient to easily recover \cref{maintheorem}: there are well-studied \emph{coarse projections} from the sphere graph to the free splitting graph, but coarse projections do not interact well with graph colorings.

In addition to the motivation from mapping class groups, Gaster, Greene, and Vlamis describe a connection between the chromatic number of the curve graph and interesting open problems in the theory of combinatorial designs. We are unaware of any analogous investigation for spheres in connect sums, but the parallel questions are natural and intriguing.

\subsection*{Notational conventions} A \emph{vertex coloring} of a graph $G$ is a function $\phi \colon V(G) \to X$ such that for every edge $(u,v)\in E(G)$, $f(u)\neq f(v)$. The cardinality $|X|$ is the size of the coloring. The \emph{chromatic number} of $G$, denoted $\chi(G)$, is the minimal size of a coloring of $G$. 

Given two functions $f,g \colon \mathbb{N}\to \mathbb{R}_{\ge 0}$, define $f\preceq g$ if there exists an absolute constant $C$ such that $f \le C\cdot g$ (often written $f\in O(g)$ elsewhere in the literature), and $f\sim g$ if $f\preceq g$ and $g\preceq f$ (often $f\in \Theta(g)$).

For $r \ge 0$ we define $M_r = M_{r, 0} = \sharp_r S^1\times S^2$ to be the connect sum of $r$ copies of $S^1\times S^2$ with the convention $M_0 = S^3$. For $n\ge 0$ we define $M_{r,n} = M_r \setminus \sqcup_k B$ to be $M_r$ with $n$ disjoint open balls removed. A sphere embedded in a 3--manifold is \emph{essential} if it does not bound a ball and \emph{non-peripheral} if it is not isotopic to the boundary. The \emph{sphere graph} of a 3--manifold M is the graph $\sph(M)$ with vertex set the isotopy classes of embedded essential non-peripheral spheres in $M$ where two classes are joint by an edge if they have disjoint representatives. For brevity, we will refer to vertices of $\sph(M)$ as \emph{spheres in $M$} when no ambiguity occurs. A \emph{cut-system} in $M_{r,n}$ is a disjoint union of embedded spheres $C$ such that $\overline{M_{r,n}\setminus C} = M_{0,n+2r}$.

\section{Kneser graphs and the lower bound}

We establish the lower bound using Gaster, Greene, and Vlamis' computation of the chromatic number for the \emph{total Kneser graph}~\cite{GGV}. Given a pair of positive integers $n, k$ with $n\ge 2k$ the \emph{Kneser graph} $KG(n,k)$ is the graph with vertices the $k$ element subsets of $\{1,\ldots, n\}$ and edges joining disjoint subsets. The \emph{total Kneser graph} is the graph whose vertices are unordered partitions of $\{1,\ldots,n\}$ into two non-empty subsets, and two partitions are joined by an edge if they are \emph{nested}, that is $(A, B)$ is joined to $(C, D)$ if one of $A$ or $B$ is a subset of $C$ or $D$. Note this condition is symmetric.  Gaster, Greene, and Vlamis compute the asymptotics of the chromatic number, which we record.

\begin{theorem}[\cite{GGV}*{Theorem 1.1}]\label{color-kneser}
\[ \chi(KG(n)) \sim n\log n. \]
\end{theorem}

The partition definition of the Kneser graph is intimately related to the structure of $\sph(M_{0,n})$. 

\begin{lemma}\label{sphere-kneser}
\[\sph(M_{0,n}) = KG(n) \setminus KG(n,1).\]
\end{lemma}

\begin{proof}
A sphere $s\in \sph(M_{0,n})$ is necessarily separating, and thus partitions the connected components of $\partial M_{0,n}$. In fact, this partition determines $s$~\cite{bering-leininger}*{Lemma 9}. Moreover, since each sphere of $\sph(M_{0,n})$ is non-peripheral, both pieces of the partition have more than one component. Fix a bijection between components of $\partial M_{0,n}$ and $\{1,\ldots, n\}$. 
The map that sends a sphere to the corresponding partition of $\{1,\ldots, n\}$ is a bijection between the vertices of $\sph(M_{0,n})$ and $KG(n)\setminus KG(n,1)$.

It remains to verify that this assignment preserves the edge relation. Suppose $a, b\in \sph(M_{0,n})$ are disjoint. Then $a$ is contained in one of the two components of $M_{0,n}\setminus b$, so the partition induced by $a$ nests with the partition induced by $b$. Suppose $a,b \in \sph(M_{0,n})$ intersect essentially, and pick representatives of their isotopy classes that intersect minimally. Let $X$ be a component of $M_{0,n}\setminus b$. Since $a$ intersects $b$ minimally, $a\cap X$ is a collection of essential disks in $X$. Therefore each component of $X\setminus a$ contains some component of $\partial M_{0,n}$ and the induced partitions are not nested.
\end{proof}

As a specific example, \cref{s5-petersen} illustrates $\sph(M_{0,5})$, and the picture can be used to verify
\[ \sph(M_{0,5}) = KG(5)\setminus KG(5,1) = KG(5,2);\]
note that $KG(5,2)$ is the well-known Petersen graph.

\begin{figure}
\centering
    \begin{tikzpicture}[every node/.style={font=\large},thick,scale=3]

\node[circle,draw=black,fill=black!10,thick, minimum size=1.6cm] (v1) at (0:1) {};
\node[circle,draw=black,fill=black!10,thick, minimum size=1.6cm] (v2) at (72:1) {};
\node[circle,draw=black,fill=black!10,thick, minimum size=1.6cm] (v3) at (144:1) {};
\node[circle,draw=black,fill=black!10,thick, minimum size=1.6cm] (v4) at (216:1) {};
\node[circle,draw=black,fill=black!10,thick, minimum size=1.6cm] (v5) at (288:1) {};

\draw[figblue] ($ (v1)+(216:0.39) $) arc (216:396:0.39) to[out=126,in=306]  node[midway,below left] {}
                    ($ (v2)+(36:0.39) $) arc  (36:216:0.39) to[out=306,in=126]
                   cycle;
\draw[figblue] ($ (v1)+(72:0.3) $) arc (72:-108:0.3) to[out=162,in=-18] ($ (v3) + (252:0.3) $) arc (252:72:0.3) to[out=-18,in=162] cycle;
\draw[figblue] ($(v1) + (288:0.36)$) arc (288:468:0.36) to[out=198,in=18]
              ($(v4) + (108:0.36)$) arc (108:288:0.36) node[midway,below left] {} to[out=18,in=198]
              cycle;
\draw[figblue] ($(v1) + (324:0.42)$) arc (-36:144:0.42) to [out=234,in=54]
            ($(v5) + (144:0.42)$) arc (144:324:0.42) to [out=54, in=234] node[midway, below right] {}
            cycle;
\draw[figyellow] ($(v2) + (288:0.42)$) arc (288:468:0.42) to[out=198,in=18] node[midway,below right] {}
            ($(v3) + (108:0.42)$) arc (108:288:0.42) to[out=18,in=198] cycle;
            
\draw[figyellow] ($(v3) + (0:0.39)$) arc (0:180:0.39) to[out=270,in=90]
              node[midway,left] {}
              ($(v4) + (180:0.39)$) arc (180:360:0.39) to[out=90,in=270]
              cycle;
\draw[figyellow] ($(v2) + (324:0.33)$) arc (-36:144:0.33) to [out=234,in=54]
            ($(v4) + (144:0.33)$) arc (144:324:0.33) to [out=54, in=234] 
cycle;
\draw[figred] ($(v3) + (36:0.33)$) arc (36:216:0.33) to [out=306,in=126]
              ($(v5) + (216:0.33)$) arc (216:396:0.33) node[midway, below] {} to [out=126,in=306]
              cycle;
\draw[figred] ($(v2) + (0:0.36)$) arc (0:180:0.36) to [out=270,in=90]
              ($(v5) + (180:0.36)$) arc (180:360:0.36) node[midway, below] {} to [out=90,in=-90]
              cycle;
\draw[figred] ($ (v5)+(72:0.3) $) arc (72:-108:0.3) to[out=162,in=-18] ($ (v4) + (252:0.3) $) arc (252:72:0.3) to[out=-18,in=162] cycle;
\end{tikzpicture}
\caption{The stereographic projection of the midsphere of $M_{0,5}$ with the deleted balls shown as removed disks. Each sphere of $\sph(M_{0,5})$ is shown as a simple closed curve of intersection with the midsphere. A pleasant exercise is to verify the intersection pattern of these spheres is the Petersen graph, 3-colored by the indicated colors.}\label{s5-petersen}
\end{figure}
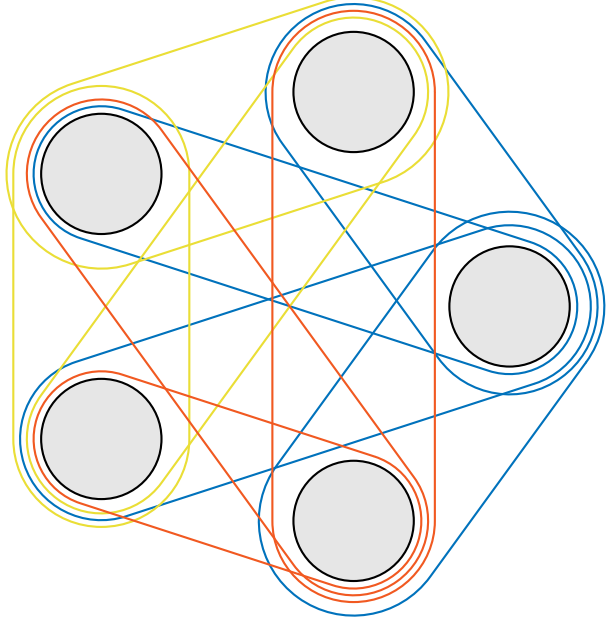

\begin{lemma}\label{lower-bound}
\[ \chi(\sph(M_{r,n})) \succeq (n+2r)\log(n+2r).\]
\end{lemma}

\begin{proof}
Fix a cut system $C$ in $M_{r,n}$. Consider the cut manifold $\overline{M_{r,n}\setminus C} = M_{0,n+2r}$. Gluing along $C$ induces a graph map $\sph(M_{0,n+2r})\to \sph(M_{r,n})$. Indeed, if an embedded sphere $a\subset M_{0,n+2r}$ bounds a ball in the glued manifold, then up to isotopy $C$ can be made disjoint from this ball, which implies $a$ is not essential. Moreover, if $a,b \subset M_{0,n+2r}$ are disjoint spheres, their images in the glued manifold are as well.

Thus $\chi(\sph(M_{r,n}))\ge \chi(\sph(M_{0,n+2r}))$, and the conclusion follows from \cref{sphere-kneser} and \cref{color-kneser}.
\end{proof}

\section{Double covers and the upper bound}

We exhibit an explicit coloring, using a double-cover construction analogous to that used by Bestvina, Bromberg, and Fujiwara in the curve complex setting~\cite{BBF}*{Lemma 5.6}.

Define $T(M)$ as the set of all of connected double covers of $M$. For a cover $\tilde{M}\in T(M)$ let $X_{\tilde{M}}(M)$ be the set of all subsets of $H_2(\tilde{M},\mathbb{Z}/2\mathbb{Z})$ of size at most 2, and define $X(M)$ to be the disjoint union of $X_{\tilde{M}}(M)$ over all $\tilde{M}\in T(M)$.
    Finally, let $F(M)$ be the set of all functions $f:T(M)\to X(M)$. The set $F(M)$ will serve as our coloring set.

\begin{lemma}\label{finite-colors}
For the manifold $M_{r}$,
\[ |F(M_r)| \preceq 2^{9r2^r}. \]
\end{lemma}

\begin{proof}
First, since $\pi_1(M_r)$ is generated by $r$ elements and the index-2 subgroups are parameterized by homomorphisms to $\mathbb{Z}/2\mathbb{Z}$, we have $|T(M_r)| \le 2^r$. 

Next, for each cover $\tilde{M}\in T(M_r)$, the homology has rank $4r-2$, thus 
\[ |X_{\tilde{M}}(M_r)| = \binom{2^{4r-2}}{1} + \binom{2^{4r-2}}{2} \preceq 2^{8r} \]
and taking a disjoint union over the $2^r$ covers we find $|X(M_r)| \preceq 2^{9r}$, from which we conclude $|F(M_r)| \preceq 2^{9r2^r}$.
\end{proof}

Given a sphere $a\in\sph(M_r)$ and a double cover $\tilde{M}$ let $\tilde{a},\tilde{a}'$ denote the two lifts of $a$ to $\tilde{M}$. For this sphere, define a function $f_a\in F(M_r)$ by
\[ f_a(\tilde{M}) = \{ [\tilde{a}],[\tilde{a}']\}, \]
that is, $f_a$ assigns to a double cover $\tilde{M}$ the $\mathbb{Z}/2\mathbb{Z}$ homology classes of the lifts of $a$ to $\tilde{M}$.

\begin{lemma}\label{upper-bound}
\[ \chi(\sph(M_r)) \preceq 2^{9r2^r}.\]
\end{lemma}

\begin{proof}
First, if $r=2$ then $\sph(M_2)$ is the Farey graph with fins~\cite{culler-vogtmann}*{Section 6}. The Farey graph is planar, so has chromatic number 4~\cites{appel-haken,debruijn-erdos}, each ``fin" joins two vertices of the Farey graph by an edge-path of length two, thus $\chi(M_2) = 4$ as well. So we suppose $r\ge 3$.

Consider the function $\phi : \sph(M_r)\to F(M_r)$ defined $\phi(a) = f_a$. We will prove $\phi$ is a coloring; the bound then follows from \cref{finite-colors}.

Let $a,b$ be adjacent spheres in $\mathcal{S}(M_r)$. 
We will construct a cover $\tilde{M}\in T(M_r)$ such that $f_a(\tilde{M})\neq f_b(\tilde{M})$. First, observe that if $a$ and $b$ are not $\mathbb{Z}/2\mathbb{Z}$ homologous in $M_r$ then for any cover $\tilde{M}$ their lifts are also not homologous, so $f_a\neq f_b$. It remains to consider the case when $[a] = [b]$.

Since $a,b$ are disjoint, there are at most three components of $M_r\setminus(a\cup b)$. 
Label the closures of the components $W_1,W_2,W_3$ and notice that each one is of the form $M_{k_i,p_i}$ with $1\leq p_i\leq 4$. Observe that for $k\ge 1$, $M_{k,p}$ contains a non-separating sphere.
We claim that for each $W_i$, $k_i \ge 1$. To verify this we consider the four possible cases with $k_i = 0$. 

\begin{enumerate}
    \item[$p_i=1$]
    In this case, $W_i$ is a ball bounded by exactly one of $a$ or $b$. 
    Thus either $a$ or $b$ is non-essential in $M_r$, contradicting our hypotheses.

    \item[$p_i=2$]
    In this case, $\partial W_i = a\cup b$, and $W_i = M_{0,2}$,
    which is homeomorphic to $S^2\times I$. 
    This implies that $a$ and $b$ are isotopic, contradicting our hypotheses. 

    \item[$p_i=3$] 
    In this case, $\partial W_i$ is three spheres.
    Without loss of generality, suppose two boundary components are identified with $b$ and one with a sphere $a$.
    We claim there is some curve $\gamma$ in $M_r$ that passes through $b$ exactly once and has $a\cap \gamma=\emptyset$.
    Given a neighborhood $N$ of $b$, select a pair of points $p_1,p_2$ in the two components of $N\setminus b$. 
    Since $N$ is path-connected and $b$ separates $N$, there exists some path $P_N$ between $p_1$ and $p_2$ that intersects $b$ exactly once.
    Since $W_i$ is path-connected and $N\setminus b\subset W_i$, there exists a path $P_c \subset W_i$ connecting $p_1,p_2$.
    By construction the union $\gamma=P_N\cup P_c$ is the desired curve. By Poincar\'{e} duality we conclude $a$ and $b$ are not homologous in $M_r$, contradicting our hypotheses. 
                
    \item[$p_i=4$] 
    In this case, there is only one component $W_1 = M_{0,4}$.
    This implies $r=2$. However, we are only considering $r\geq 3$.
    
\end{enumerate}

Thus, each $W_i$ contains some non-separating sphere $\sigma_i$.
Construct a double cover $\tilde{M}$ by cutting open two copies of $M_r$ along the collection of $\sigma_i$ and gluing crosswise.
Notice that if $a\cup b$ does not separate $M_r$, then we can show that $[a]\neq[b]$ via a similar curve construction.
Therefore $a\cup b$ separates $M_r$. 

We will construct a closed loop $\gamma$ in $\tilde{M}$ that intersects both lifts of $a$, $\tilde{a}$ and $\tilde{a}'$ exactly once and disjoint from both lifts of $b$.

As noted in the case $p_i = 3$, if $a$ appears twice in the boundary of one complementary component, $[a]\neq [b]$. So, without loss of generality suppose $a$ is in the boundary of $W_1$ and $W_2$. Since each $W_i$ is path connected and the spheres $\sigma_i$ are non-separating, there is a loop $\gamma_0$ that intersects $\sigma_1$ and $\sigma_2$ exactly once, intersecting $a$ twice. By construction, there is a loop $\gamma$ in $\tilde{M}$ covering $\gamma_0$, disjoint from both lifts of $b$ which intersects $\tilde{a}$ and $\tilde{a}'$ exactly once (see \cref{double-cover-gamma}).
\end{proof}

\begin{figure}
\centering\labellist
\small\hair 2pt
 \pinlabel {{\color{figred}$a$}} [ ] at 107 109
 \pinlabel {{\color{figblue}$b$}} [ ] at 107 36
 \pinlabel {{\color{figyellow}$\sigma_1$}} [ ] at 40 36
 \pinlabel {{\color{figyellow}$\sigma_2$}} [ ] at 175 36
 \pinlabel {{\color{figmagenta}$\gamma_0$}} [ ] at 40 85
 \pinlabel {{\color{figyellow}$\sigma_1$}} [ ] at 257 36
 \pinlabel {{\color{figred}$a$}} [ ] at 293 36
 \pinlabel {{\color{figyellow}$\sigma_2$}} [ ] at 328 36
 \pinlabel {{\color{figblue}$b$}} [ ] at 362 36
 \pinlabel {{\color{figmagenta}$\gamma_0$}} [ ] at 257 85
\endlabellist
\centering
\includegraphics[width=\textwidth]{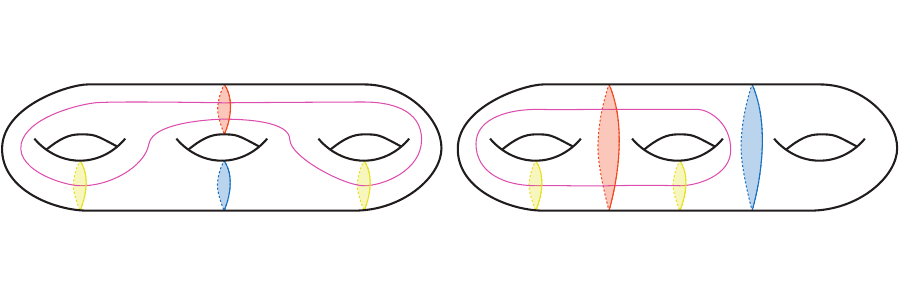}
\caption{Two possible cases for the construction of $\gamma_0$ in $M_3$, here depicted as one-half of a Heegaard splitting. The spheres $a, b,\sigma_1,\sigma_2$ are the doubles of the indicated disks.}
\label{double-cover-gamma}
\end{figure}
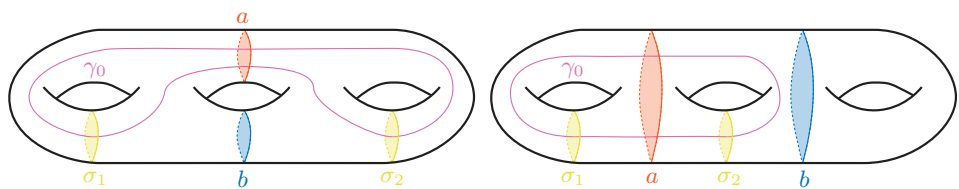

\end{document}